\documentclass[12pt]{article}

\usepackage{amsfonts,amsmath,amsthm,amssymb,amscd}
\usepackage[dvips]{graphicx}

\theoremstyle{definition}
\newtheorem{theorem}{Theorem}
\newtheorem{lemma}[theorem]{Lemma}
\newtheorem{proposition}[theorem]{Proposition}

\numberwithin{equation}{section}
\numberwithin{theorem}{section}

\begin{document}

\begin{center}
{\bf{\Large On certain quaternary quadratic forms}}
\end{center}

\begin{center}
By Kazuhide Matsuda
\end{center}

\begin{center}
Department of Engineering Science, Niihama National College of Technology,\\
7-1 Yagumo-chou, Niihama, Ehime, Japan, 792-8580 \\
E-mail: matsuda@sci.niihama-nct.ac.jp  \\
Fax: 0897-37-7809 
\end{center}

{\bf Abstract}
In this paper, 
we determine all the positive integers $a, b$ and $c$ such that 
every nonnegative integer can be represented as 
$$
f^{a,b}_c(x,y,z,w)=ax^2+by^2+c(z^2+zw+w^2) \,\, \textrm{with} \,\,x,y,z,w\in\mathbb{Z}.
$$
Furthermore, 
we prove that $f^{a,b}_c$ can represent all the nonnegative integers if it represents $n=1,2,3,5,6,10.$ 
\newline
{\bf Key Words:} quaternary quadratic forms; theta functions; ``The 290-theorem''. 
\newline
{\bf MSC(2010)}  14K25;  11E25

\section{Introduction}
\quad
Throughout this paper, 
set $\mathbb{N}=\{1,2,3,\ldots\}$ and $\mathbb{N}_0=\{0,1,2,3,\ldots\}.$ 
$\mathbb{Z}$ denotes the set of rational integers. 
In addition, 
the triangular numbers are $t_x=x(x+1)/2,\,\, (x\in\mathbb{N}_0)$ and  
the squares are $y^2, \,\,(y\in\mathbb{Z}).$ 
Furthermore, $a,b$ and $c$ are fixed positive integers with $1\le a\le b.$
\par
Lagrange proved that 
every $n\in\mathbb{N}_0$ is a sum of four squares. 
Noted is that Jacobi \cite{Jacobi} showed this fact using the elliptic function theory. 
\par
Ramanujan \cite{Ramanujan} determined all the positive integers 
$a,b,c$ and $d$ such that every $n\in\mathbb{N}_0$ is represented as 
$ax^2+by^2+cz^2+du^2$ with $x,y,z,u\in\mathbb{Z}.$ 
He proved that there exist fifty-four such quadruples $(a,b,c,d)$ with $1\le a \le b \le c \le d.$ 
A proof of this fact is contained in Dickson \cite{Dickson-2}. 
\par 
The aim of this paper is to determine all the positive integers 
$a, b$ and $c$ with $a\le b$ 
such that 
every $n\in\mathbb{N}_0$ can be represented as 
$$
f^{a,b}_c(x,y,z,w)=ax^2+by^2+c(z^2+zw+w^2)
$$ 
with $x,y,z,w\in\mathbb{Z}.$ 
Furthermore, 
we prove that $f^{a,b}_c$ can represent all the nonnegative integers if it represents $n=1,2,3,5,6,10.$ 
Our main theorem is as follows: 
\begin{theorem}
\label{thm:main}
{\it
For positive integers $a,b$ and $c$ with $a\le b,$ 
set
$$
f^{a,b}_c(x,y,z,w)=ax^2+by^2+c(z^2+zw+w^2), \,\,(x,y,z,w\in\mathbb{Z}).
$$ 
\begin{enumerate}\itemsep=0pt
\item[$(1)$] 
$f^{a,b}_c$ can represent any nonnegative integers 
if and only if  
\begin{equation*}
(a,b,c)=
\begin{cases}
(1,b,1)   &(b=1,2,3,4,5,6), \\
(2,b,1)  &(b=2,3,4,5,6,7,8,9,10),  \\
(1,b,2)  &(b=1,2,3,4,5),  \\
(1,2,3),   &  \\
(1,2,4).   &
\end{cases}
\end{equation*}
\item[$(2)$] 
If $f^{a,b}_c$ represents $n=1,2,3,5,6,10,$ 
then it can represent all the nonnegative integers. 
\end{enumerate}
}
\end{theorem}

\subsection*{Remarks}
Theorem \ref{thm:main} (2) is the restriction of ``The 290-theorem'' to the forms $ax^2+by^2+c(z^2+zw+w^2).$
Bhargava and Hanke \cite{Bhargava-Hanke} proved ``The 290-theorem''. 
\begin{theorem}(``The 290-theorem'')
\label{thm:290}
{\it
If a positive-definite quadratic form with integer coefficients 
represents the 29 integers
\begin{align*}
&1,2,3,5,6,7,10,13,14,15,17,19,21,22,23,26,  \\
&29,30,31,34,35,37,42,58,93,110,145,203,290,
\end{align*}
then it represents all positive integers.
}
\end{theorem}

For the proof of Theorem \ref{thm:main}, 
following Ramanujan, 
we introduce 
\begin{equation*}
\varphi(q)=\sum_{n\in\mathbb{Z}} q^{n^2},  \,\,
\psi (q) =\sum_{n=0}^{\infty} q^{\frac{n(n+1)}{2}}, \,\,
a(q)=\sum_{m,n\in\mathbb{Z}} q^{m^2+mn+n^2}, \,\,(q\in\mathbb{C}, \,\,|q|<1), 
\end{equation*}
and use the following identities:
\begin{align}
&a(q)=a(q^4)+6q\psi (q^2)\psi (q^6).                                               \label{eqn:tri(1,3)}   \\ 
&\varphi(q)\varphi(q^3)=a(q^4)+2q\psi (q^2)\psi (q^6).                           \label{eqn:tri(1,3)-(2)}
\end{align}
For the proof of these formulas, 
see Berndt \cite[pp. 232]{Berndt} and Hirschhorn et al. \cite{Hirschhorn-Garvan-Borwein}.

\subsection*{Acknowledgments}

We are grateful to Professor Williams for informing us the ``The 290-theorem'' and, for his useful suggestions.

\section{Notations and preliminaries}
For the fixed positive integers $a,c$ and each $n\in\mathbb{N}_0,$  
we define
$$
f^a_c(x,y,z)=ax^2+c(y^2+yz+z^2),
$$
and 
$$
A^a_c(n)=\sharp
\left\{
(x,y,z)\in\mathbb{Z}^3 \, | \, n=f^a_c(x,y,z)
\right\}. 
$$
Moreover, for the fixed positive integers $\alpha, \beta, \gamma$ and each $n\in\mathbb{N}_0,$ 
we set 
\begin{align*}
r_{\alpha,\beta}(n)=&\sharp
\left\{
(x,y)\in\mathbb{Z}^2 \, | \,n=\alpha x^2+\beta y^2
\right\}, \\
t_{\alpha,\beta}(n)=&\sharp
\left\{
(x,y)\in\mathbb{N}_0^2 \, | \,n=\alpha t_x+\beta t_y 
\right\}, \\
r_{\alpha,\beta,\gamma}(n)=&\sharp
\left\{
(x,y,z)\in\mathbb{Z}^3 \, | \,n=\alpha x^2+\beta y^2+\gamma z^2 
\right\},  \\
m_{\alpha\text{-}\beta, \gamma}(n)
=&\sharp
\left\{
(x,y,z)\in\mathbb{Z}\times\mathbb{N}_0^2 \, | \,n=\alpha x^2+\beta t_y+\gamma t_z 
\right\}. \\
\end{align*}

\section{The case where $c=1$}

\begin{lemma}
\label{lem:c=1}
Suppose that $a\le b, \,\, c=1$ and $f^{a,b}_c(x,y,z,w), \,\,(x,y,z,w\in\mathbb{Z})$ 
represent any nonnegative integers $n\in\mathbb{N}_0.$ 
Then, $a=1,2.$ 
\end{lemma}

\begin{proof}
The lemma follows from the fact that 
$n=2$ cannot be written as $z^2+zw+w^2$ with $z,w\in\mathbb{Z}.$ 
\end{proof}

\subsection{The case where $a=1$}

We use the following result of Dickson \cite[pp. 112-113]{Dickson-2}:

\begin{lemma}
\label{lem-1,1,3}
{\it
$n\in\mathbb{N}_0$ can be written as $x^2+y^2+3z^2$ with $x,y,z\in\mathbb{Z}$ 
if and only if $n\neq 9^k(9l+6), \,\, (k,l\in\mathbb{N}_0).$ 
}
\end{lemma}

By Lemma \ref{lem-1,1,3}, we have the following proposition:

\begin{proposition}
\label{prop:a=1,c=1}
{\it
$n\in\mathbb{N}_0$ can be written as $x^2+(y^2+yz+z^2)$ with $x,y,z\in\mathbb{Z}$ 
if and only if $n\neq 9^k(9l+6), \,\,(k,l\in\mathbb{N}_0).$
}
\end{proposition}

\begin{proof}
Multiplying both sides of equations (\ref{eqn:tri(1,3)}) and (\ref{eqn:tri(1,3)-(2)}) by $\varphi(q),$ 
we have 
\begin{align}
\varphi(q)a(q)=&\varphi(q)a(q^4)+6q\varphi(q)\psi (q^2)\psi (q^6),  \\
\varphi(q)^2\varphi(q^3)=&\varphi(q)a(q^4)+2q\varphi(q)\psi (q^2)\psi (q^6). 
\end{align}
which implies that 
\begin{align*}
\sum_{n=0}^{\infty} A^1_1(n) q^n=&\sum_{n=0}^{\infty} A^1_4(n) q^n +6q \sum_{N=0}^{\infty} m_{1\text{-}2,6}(N)q^N,  \\
\sum_{n=0}^{\infty} r_{1,1,3}(n) q^n=&\sum_{n=0}^{\infty} A^1_4(n) q^n +2q \sum_{N=0}^{\infty} m_{1\text{-}2,6}(N)q^N.                                         
\end{align*}
Therefore, 
it follows that
\begin{align*}
n\neq 9^k(9l+6), \,\,(k,l\in\mathbb{N}_0) &\Longleftrightarrow r_{1,1,3}(n)>0  \\
                                                     &\Longleftrightarrow A^1_4(n)>0 \,\textrm{or} \, m_{1\text{-}2,6}(n-1)>0,  \\
                                                     &\Longleftrightarrow A^1_1(n)>0,
\end{align*}
which proves the proposition. 
\end{proof}

By Proposition \ref{prop:a=1,c=1}, we obtain the following theorem:

\begin{theorem}
\label{thm:a=1,c=1}
{\it
$f^{1,b}_1$ 
can represent any nonnegative integers $n\in\mathbb{N}_0$ 
if and only if $b=1,2,3,4,5,6.$
}
\end{theorem}

\begin{proof}
The ``only if'' direction follows from Proposition \ref{prop:a=1,c=1}. 
In order to establish the ``if'' direction, 
we have only to prove that $f^{1,b}_1$ represents $n=9l+6, \,\,(l\in\mathbb{N}_0).$ 
\par
When $b=1,2,3,4,5,$ taking $y=1,$ 
we have 
$$
n-by^2=9l+6-b\cdot1^2\equiv 5, 4,3,2, 1 \,\,\mathrm{mod}\,9,
$$
which can be written as $x^2+(z^2+zw+w^2)$ with $x,z\in\mathbb{Z}.$ 
\par
Suppose that $b=6.$ 
If $n=6, 15,$ or $l\equiv 2, 8 \,\mathrm{mod} \,9,$ 
taking $y=1,$ 
we obtain 
$$
n-6\cdot 1^2=9l, \,\,(l=0,1 \,\textrm{or} \,l\equiv 2, 8 \,\mathrm{mod} \,9),
$$
which can be represented as $x^2+(z^2+zw+w^2)$ with $x,z\in\mathbb{Z}.$ 
\par
If $l\ge 2$ and $l \not\equiv 2, 8 \,\mathrm{mod} \,9,$ 
taking $y=2,$ 
we have 
$$
n-6\cdot 2^2=9l+6-24=9(l-2),
$$
which can be written as $x^2+(z^2+zw+w^2)$ with $x,z\in\mathbb{Z}.$
\end{proof}

\subsection{The case where $a=2$}

We use the following result of Dickson \cite{Dickson-1}: 

\begin{lemma}
\label{lem-1,2,3}
{\it
\quad
\begin{enumerate}\itemsep=0pt
\item[$(1)$] 
$n\in\mathbb{N}_0$ can be written as 
$x^2+2y^2+3z^2$ with $x,y,z\in\mathbb{Z}$ if and only if 
$n\neq  4^k(16l+10), \,\,(k,l\in\mathbb{N}_0).$ 
\item[$(2)$] 
$n\in\mathbb{N}_0$ can be written as 
$x^2+2(y^2+yz+z^2)$ with $x,y,z\in\mathbb{Z}$ if and only if 
$n\neq 4^k(8l+5), \,\,(k,l\in\mathbb{N}_0).$ 
\end{enumerate}
}
\end{lemma}

By Lemma \ref{lem-1,2,3}, we have the following proposition:

\begin{proposition}
\label{prop:a=2,c=1}
{\it
$n\in\mathbb{N}_0$ can be written as $2x^2+(y^2+yz+z^2)$ with $x,y,z\in\mathbb{Z}$ 
if and only if $n\neq4^k(16l+10), \,\,(k,l\in\mathbb{N}_0).$
}
\end{proposition}

\begin{proof}
Multiplying both sides of equations (\ref{eqn:tri(1,3)}) and (\ref{eqn:tri(1,3)-(2)}) by $\varphi(q^2),$ 
we have 
\begin{align}
\varphi(q^2)a(q)=&\varphi(q^2)a(q^4)+6q\varphi(q^2)\psi (q^2)\psi (q^6),  \\
\varphi(q) \varphi(q^2) \varphi(q^3)=&\varphi(q^2)a(q^4)+2q\varphi(q^2)\psi (q^2)\psi (q^6). 
\end{align}
which implies that 
\begin{align*}
\sum_{n=0}^{\infty} A^2_1(n) q^n=&\sum_{N=0}^{\infty} A^1_2(N) q^{2N} +6q \sum_{N=0}^{\infty} m_{1\text{-}1,3}(N)q^{2N},  \\
\sum_{n=0}^{\infty} r_{1,2,3}(n) q^n=&\sum_{N=0}^{\infty} A^1_2(N) q^{2N} +2q \sum_{N=0}^{\infty} m_{1\text{-}1,3}(N)q^{2N}.                                         
\end{align*}
\par
Suppose that $n$ is even and $n=2N.$ 
Therefore, 
it follows that
\begin{align*}
n\neq 4^k(16l+10), \,\,(k,l\in\mathbb{N}_0) &\Longleftrightarrow r_{1,2,3}(n)>0  \\
                                                     &\Longleftrightarrow A^1_2(N)>0  \\
                                                     &\Longleftrightarrow  A^2_1(n)>0. 
\end{align*}
\par
Suppose that $n$ is odd and $n=2N+1.$ 
Therefore, 
it follows that
\begin{align*}
n\neq 4^k(16l+10), \,\,(k,l\in\mathbb{N}_0) &\Longleftrightarrow r_{1,2,3}(n)>0  \\
                                                     &\Longleftrightarrow m_{1\text{-}1,3}(N)>0  \\
                                                     &\Longleftrightarrow  A^2_1(n)>0. 
\end{align*}
\end{proof}

By Proposition \ref{prop:a=2,c=1}, we obtain the following theorem:

\begin{theorem}
\label{thm:a=2,c=1}
{\it
Suppose that $2\le b.$  
Then, 
$f^{2,b}_1$ 
can represent any nonnegative integers $n\in\mathbb{N}_0$ 
if and only if $b=2,3,4,5,6,7,8,9,10.$
}
\end{theorem}

\begin{proof}
The ``only if'' direction follows from Proposition \ref{prop:a=2,c=1}. 
In order to establish the ``if'' direction, 
we have only to prove that $f^{2,b}_1$ represents $n=16l+10, \,\,(l\in\mathbb{N}_0).$ 
\par
When $b\neq 2,10,$ taking $y=1,$ 
we have 
$$
n-by^2=16l+10-b\cdot1^2\not\equiv 0,8, 10  \,\,\mathrm{mod}\,16,
$$
which can be written as $2x^2+(z^2+zw+w^2)$ with $x,z\in\mathbb{Z}.$ 
\par
When $b=2,$ 
taking $y=2,$ we obtain 
$$
n-2\cdot 2^2=16l+10-8=16l+2, 
$$
which can be represented as $2x^2+(z^2+zw+w^2)$ with $x,z\in\mathbb{Z}.$ 
\par
We last suppose that $b=10.$ 
If $n=10, 26,$ 
taking $y=1,$ 
we obtain 
$$
n-10\cdot 1^2=0, 16, 
$$
which can be written as $2x^2+(z^2+zw+w^2)$ with $x,z\in\mathbb{Z}.$ 
\par
If $l\ge 2,$ 
taking $y=2,$ 
we have 
$$
n-10\cdot 2^2=16l+10-40=16(l-2)+2,
$$
which can be represented as $2x^2+(z^2+zw+w^2)$ with $x,z\in\mathbb{Z}.$
\end{proof}

\subsection{Summary}
By Theorems \ref{thm:a=1,c=1}, \ref{thm:a=2,c=1} and their proofs, 
we obtain 
\begin{theorem}
\label{thm:c=1-summary}
{\it
Let $a,b$ be positive integers with $a\le b.$ 
\begin{enumerate}\itemsep=0pt
\item[$(1)$] 
$f^{a,b}_1$ can represent any nonnegative integers if and only if $(a,b)$ are given by 
\begin{equation*}
(a,b)=
\begin{cases}
(1,b)   &(b=1,2,3,4,5,6), \\
(2,b)  &(b=2,3,4,5,6,7,8,9,10).
\end{cases}
\end{equation*}
\item[$(2)$]
If $f^{a,b}_1$ represents $n=1,2, 6,10,$ 
it can represent all the nonnegative integers. 
\end{enumerate}
}
\end{theorem}

Furthermore, we obtain the following theorem:

\begin{theorem}
{\it
For fixed positive integers, $a,c,$ set
$$
f^a_c(x,y,z)=ax^2+c(y^2+yz+z^2) \,\,\text{with} \,\,x,y,z\in\mathbb{Z}.
$$
There exists no positive integers, $a,c$ such that $f^a_c$ can represent every natural number. 
}
\end{theorem}

\begin{proof}
Suppose that there exist such positive integers, $a,c.$ 
Taking $n=1,$ we have $a=1$ or $c=1.$ 
\par
Suppose that $a=1.$ The choice $n=2$ implies that $c=1,2.$ 
On the other hand, 
if $(a,c)=(1,1), (1,2),$ by Proposition \ref{prop:a=1,c=1} and Lemma \ref{lem-1,2,3}, 
we see that there exist positive integers which cannot be expressed by $f^1_1$ or $f^1_2,$ 
which is contradiction. 
\par
Suppose that $c=1.$ The choice $n=2$ implies that $a=1,2.$ 
If $(a,c)=(1,1),(2,1),$ by Propositions \ref{prop:a=1,c=1} and \ref{prop:a=2,c=1}, 
we see that there exist positive integers which cannot be expressed by $f^1_1$ or $f^1_2,$ 
which is contradiction. 
\end{proof}

\subsubsection*{Remark}
In \cite[pp.104]{Dickson-2}, 
Dickson proved that 
there exist no positive integers, $a,b,c,$ 
such that $ax^2+by^2+cz^2, \,(x,y,z\in\mathbb{Z})$ can represent all positive integers.

\section{The case where $c=2$}

Noting that $a=1$ 
if $f^{a,b}_2$ can represent any nonnegative integers, by Lemma \ref{lem-1,2,3} (2), 
we obtain the following theorem:

\begin{theorem}
\label{thm:a=1,c=2}
{\it
Suppose that $1\le a \le b.$ 
Then, 
$f^{a,b}_2$ 
can represent any nonnegative integers $n\in\mathbb{N}_0$ 
if and only if $a=1$ and $b=1,2,3,4,5.$
}
\end{theorem}

\begin{proof}
The ``only if'' direction follows from Lemma \ref{lem-1,2,3} (2). 
In order to establish the ``if'' direction, 
we have only to prove that  $f^{1,b}_2$ represents $n=8l+5, \,\,(l\in\mathbb{N}_0).$ 
\par
We first treat the case where $b=1.$ 
Taking $y=2,$ 
we obtain 
$$
n-1\cdot 2^2=8l+5-4=8l+1, 
$$
which can be represented as $x^2+2(z^2+zw+w^2)$ with $x,z\in\mathbb{Z}.$
\par
When $b=2,3,4,$ taking $y=1,$ 
we have 
$$
n-by^2=8l+5-b\cdot1^2\equiv 3,2, 1 \,\,\mathrm{mod}\,8,
$$
which can be written as $x^2+2(z^2+zw+w^2)$ with $x,z\in\mathbb{Z}.$ 
\par
We last suppose that $b=5.$ 
If $l=0,1,$  
taking $y=1,$ 
we obtain 
$$
n-5\cdot 1^2=8l+5-5=0,8, 
$$
which can be represented as $x^2+2(z^2+zw+w^2)$ with $x,z\in\mathbb{Z}.$ 
\par
If $l\ge 2,$ 
taking $y=2,$ 
we have 
$$
n-5\cdot 2^2=8l+5-20=8(l-2)+1,
$$
which can be written as $x^2+2(z^2+zw+w^2)$ with $x,z\in\mathbb{Z}.$
\end{proof}

By the proof of Theorem \ref{thm:a=1,c=2}, 
we obtain 
\begin{theorem}
\label{thm:a=1,c=2-(2)}
{\it
If $f^{a,b}_2$ represents $n=1,5,$ 
it can represent all the nonnegative integers. 
}
\end{theorem}

\section{The case where $c=3$}

We use the result of Dickson \cite[pp. 112-113]{Dickson-2}:

\begin{lemma}
\label{lem-1,3,9}
{\it
$n\in\mathbb{N}_0$ can be written as $x^2+3y^2+9z^2$ with $x,y,z\in\mathbb{Z}$ 
if and only if $n\neq 3l+2, \, 9^k(9l+6), \,\, (k,l\in\mathbb{N}_0).$ 
}
\end{lemma}

By Lemma \ref{lem-1,3,9}, 
we have the following proposition:

\begin{proposition}
\label{prop:a=1,c=3}
{\it
$n\in\mathbb{N}_0$ can be written as $x^2+3(y^2+yz+z^2)$ with $x,y,z\in\mathbb{Z}$ 
if and only if $n\neq 3l+2, \, 9^k(9l+6), \,\,(k,l\in\mathbb{N}_0).$
}
\end{proposition}

\begin{proof}
Replacing $q$ by $q^3$ in equations (\ref{eqn:tri(1,3)}) and (\ref{eqn:tri(1,3)-(2)}), 
we have 
\begin{align*}
a(q^3)=&a(q^{12})+6q^3\psi (q^6)\psi (q^{18}),   \\
\varphi(q^3)\varphi(q^9)=&a(q^{12})+2q^3\psi (q^6)\psi (q^{18}). 
\end{align*}
Multiplying both sides of these equations by $\varphi(q),$ 
we have 
\begin{align}
\varphi(q)a(q^3)=&\varphi(q)a(q^{12})+6q^3\varphi(q)\psi (q^6)\psi (q^{18}),  \\
\varphi(q)\varphi(q^3)\varphi(q^9)=&\varphi(q)a(q^{12})+2q^3\varphi(q)\psi (q^6)\psi (q^{18}). 
\end{align}
which implies that 
\begin{align*}
\sum_{n=0}^{\infty} A^1_3(n) q^n=&\sum_{n=0}^{\infty} A^1_{12}(n) q^n +6q^3 \sum_{N=0}^{\infty} m_{1\text{-}6,18}(N)q^N,  \\
\sum_{n=0}^{\infty} r_{1,3,9}(n) q^n=&\sum_{n=0}^{\infty} A^1_{12}(n) q^n +2q^3 \sum_{N=0}^{\infty} m_{1\text{-}6,18}(N)q^N.                                         
\end{align*}
Therefore, 
it follows that
\begin{align*}
n\neq 3l+2, \, 9^k(9l+6), \,\,(k,l\in\mathbb{N}_0) &\Longleftrightarrow r_{1,3,9}(n)>0  \\
                                                     &\Longleftrightarrow A^1_{12}(n)>0 \,\textrm{or} \, m_{1\text{-}6,18}(n-3)>0,  \\
                                                     &\Longleftrightarrow A^1_3(n)>0,
\end{align*}
which proves the proposition. 
\end{proof}

Noting that $a=1$ 
if $f^{a,b}_3$ can represent any nonnegative integers, 
we obtain the following theorem:

\begin{theorem}
\label{thm:a=1,c=3}
{\it
$f^{a,b}_3$ 
can represent any nonnegative integers $n\in\mathbb{N}_0$ 
if and only if $a=1$ and $b=2.$
}
\end{theorem}

\begin{proof}
Let us first prove the ``only if'' direction. 
Proposition \ref{prop:a=1,c=3} implies that $b=1,2.$ 
Assume that $b=1$ and 
$$
n=6=x^2+y^2+3(z^2+zw+w^2) \,\,\textrm{with} \,\,x,y,z,w\in\mathbb{Z}, 
$$ 
which implies that $z^2+zw+w^2=1, \neq 0,2.$ 
It then follows that $x^2+y^2=3,$ which is impossible. 
\par
In order to establish the ``if'' direction, 
we have only to prove that $f^{1,b}_3$ can represent $n=3l+2, \,9l+6, \,\,(l\in\mathbb{N}_0).$ 
\par
Suppose that $n=3l+2.$ 
We first treat the case where $l\equiv 0 \,\mathrm{mod} \,3.$ 
It is obvious that $f^{1,2}_3$ can represent $n=2.$ 
When $l\ge 1$ and $l\equiv 0 \,\mathrm{mod} \,3,$ 
we obtain 
$$
n-2\cdot 2^2=3l+2-8=3(l-2), 
$$
which can be written as $x^2+3(z^2+zw+w^2)$ with $x,z\in\mathbb{Z}.$
\par
When $l\equiv 1 \,\mathrm{mod} \,3,$ taking $y=1,$ 
we have 
$$
n-2y^2=3l+2-2=3l, 
$$
which can be written as $x^2+3(z^2+zw+w^2)$ with $x,z\in\mathbb{Z}.$ 
\par
Assume that $l\equiv 2 \,\mathrm{mod} \,3.$ 
For $l=2, 5, 8,$ taking $y=2,$ 
we have 
$$
n-2\cdot 2^2=3l+2-2\cdot 2^2=0, 9, 9\cdot 2,
$$ 
which can be represented as $x^2+3(z^2+zw+w^2)$ with $x,z\in\mathbb{Z}.$
If $l\ge 11$ and $l=3L+2, \,\,(L\in\mathbb{N}_0),$ taking $y=4,$ 
we have 
$$
n-2\cdot 4^2=3l+2-2\cdot 4^2=3\{3(L-3)+1\}, 
$$
which can be written as $x^2+3(z^2+zw+w^2)$ with $x,z\in\mathbb{Z}.$
\par
We last suppose that $n=9l+6.$ 
Taking $y=1,$ 
we have 
$$
n-2y^2=9l+6-2=9l+4 \equiv 1 \,\mathrm{mod} \, 3,  
$$
which can be represented as $x^2+3(z^2+zw+w^2)$ with $x,z\in\mathbb{Z}.$ 
\end{proof}

By the proof of Theorem \ref{thm:a=1,c=3}, 
we obtain 
\begin{theorem}
\label{thm:a=1,c=3-(2)}
{\it
If $f^{a,b}_3$ represents $n=1,2,6,$ 
it can represent all the nonnegative integers. 
}
\end{theorem}

\section{The case where $c=4$}

We use the following result of Dickson \cite[pp. 112-113]{Dickson-2}:

\begin{lemma}
\label{lem-1,4,12}
{\it
$n\in\mathbb{N}_0$ can be written as $x^2+4y^2+12z^2$ with $x,y,z\in\mathbb{Z}$ 
if and only if $n\neq 4l+2, \, 4l+3, \, 9^k(9l+6), \,\, (k,l\in\mathbb{N}_0).$ 
}
\end{lemma}

By Lemma \ref{lem-1,4,12}, we have the following proposition:

\begin{proposition}
\label{prop:a=1,c=4}
{\it
$n\in\mathbb{N}_0$ can be written as $x^2+4(y^2+yz+z^2)$ with $x,y,z\in\mathbb{Z}$ 
if and only if $n\neq 4l+2, \, 4l+3, \, 9^k(9l+6), \,\, (k,l\in\mathbb{N}_0).$
}
\end{proposition}

\begin{proof}
Replacing $q$ by $q^4$ in equations (\ref{eqn:tri(1,3)}) and (\ref{eqn:tri(1,3)-(2)}), 
we have 
\begin{align*}
a(q^4)=&a(q^{16})+6q^4\psi (q^8)\psi (q^{24}),   \\
\varphi(q^4)\varphi(q^{12})=&a(q^{16})+2q^4\psi (q^8)\psi (q^{24}). 
\end{align*}
Multiplying both sides of these equations by $\varphi(q),$ 
we have 
\begin{align}
\varphi(q)a(q^4)=&\varphi(q)a(q^{16})+6q^4\varphi(q)\psi (q^8)\psi (q^{24}),  \\
\varphi(q)\varphi(q^4)\varphi(q^{12})=&\varphi(q)a(q^{16})+2q^4\varphi(q)\psi (q^8)\psi (q^{24}). 
\end{align}
which implies that 
\begin{align*}
\sum_{n=0}^{\infty} A^1_4(n) q^n=&\sum_{n=0}^{\infty} A^1_{16}(n) q^n +6q^4 \sum_{N=0}^{\infty} m_{1\text{-}8,24}(N)q^N,  \\
\sum_{n=0}^{\infty} r_{1,4,12}(n) q^n=&\sum_{n=0}^{\infty} A^1_{16}(n) q^n +2q^4 \sum_{N=0}^{\infty} m_{1\text{-}8,24}(N)q^N.                                         
\end{align*}
Therefore, 
it follows that
\begin{align*}
n\neq 4l+2, \, 4l+3, \, 9^k(9l+6), \,\, (k,l\in\mathbb{N}_0) &\Longleftrightarrow r_{1,4,12}(n)>0  \\
                                                     &\Longleftrightarrow A^1_{16}(n)>0 \,\textrm{or} \, m_{1\text{-}8,24}(n-4)>0,  \\
                                                     &\Longleftrightarrow A^1_4(n)>0,
\end{align*}
which proves the proposition. 
\end{proof}

Noting that $a=1, b=2$ 
if $f^{a,b}_4$ can represent any nonnegative integers, 
we obtain the following theorem:

\begin{theorem}
\label{thm:a=1,c=4}
{\it
$f^{a,b}_4$ 
can represent any nonnegative integers $n\in\mathbb{N}_0$ 
if and only if $a=1$ and $b=2.$
}
\end{theorem}

\begin{proof}
The ``only if'' direction is obvious. 
In order to establish the ``if'' direction, 
we have only to prove that $f^{1,2}_4$ represents $n=4l+2, \, 4l+3, \,9l+6, \,\,(l\in\mathbb{N}_0).$ 
\par
Suppose that $n=4l+2.$ 
If $l \not\equiv0,6 \,\mathrm{mod} \,9,$ 
taking $y=1,$ we have 
$$
n-2y^2=4l+2-2=4l, 
$$
which can be represented as $x^2+4(z^2+zw+w^2)$ with $x,z\in\mathbb{Z}.$ 
If $\l \equiv0,6 \,\mathrm{mod} \,9,$ taking $y=3,$ 
we obtain 
$$
n-2y^2=4l+2-2\cdot 3^2=4(l-4),
$$
which can be represented as $x^2+4(z^2+zw+w^2)$ with $x,z\in\mathbb{Z}.$ 
\par
Suppose that $n=4l+3.$ 
If $l \not\equiv2,8 \,\mathrm{mod} \,9,$ 
taking $y=1,$ we have 
$$
n-2y^2=4l+3-2=4l+1\not\equiv 0,6 \,\mathrm{mod} \,9, 
$$
which can be represented as $x^2+4(z^2+zw+w^2)$ with $x,z\in\mathbb{Z}.$ 
If $\l \equiv2,8 \,\mathrm{mod} \,9,$ taking $y=3,$ 
we obtain 
$$
n-2y^2=4l+3-2\cdot 3^2=4(l-4)+1\equiv 2, 8 \,\mathrm{mod} \, 9,
$$
which can be represented as $x^2+4(z^2+zw+w^2)$ with $x,z\in\mathbb{Z}.$ 
It is easy to check that 
$f^{1,2}_4$ can represent $n=11.$ 
\par
Suppose that $n=9l+6.$ 
If $\l\equiv 0 \,\mathrm{mod} \,4$ and $l=4L, \,\,(L\in\mathbb{N}_0),$ 
taking $y=5,$  
we have 
$$
n-2y^2=9l+6-2\cdot 5^2=4\{9(L-2)+7\}, 
$$
which can be written as $x^2+4(z^2+zw+w^2)$ with $x,z\in\mathbb{Z}.$ 
It is easy to check that $f^{1,2}_4$ can represent $n=6,42.$
\par
If $\l\equiv 1 \,\mathrm{mod} \,4,$ 
taking $y=1,$ 
we have 
$$
n-2y^2=9l+6-2\cdot 1^2=9l+4\equiv 1 \,\mathrm{mod} \, 4, 
$$
which can be represented as $x^2+4(z^2+zw+w^2)$ with $x,z\in\mathbb{Z}.$
\par
If $\l\equiv 2 \,\mathrm{mod} \,4$ and $l=4L+2, \,\,(L\in\mathbb{N}_0),$ 
taking $y=4,$  
we have 
$$
n-2y^2=9l+6-2\cdot 4^2=4(9L-2), 
$$
which can be written as $x^2+4(z^2+zw+w^2)$ with $x,z\in\mathbb{Z}.$ 
It is easy to check that $f^{1,2}_4$ can represent $n=24.$
\par
If $\l\equiv 3 \,\mathrm{mod} \,4,$ 
taking $y=2,$ 
we have 
$$
n-2y^2=9l+6-2\cdot 2^2=9l-2\equiv 1 \,\mathrm{mod} \, 4, 
$$
which can be represented as $x^2+4(z^2+zw+w^2)$ with $x,z\in\mathbb{Z}.$
\end{proof}

By the proof of Theorem \ref{thm:a=1,c=4}, 
we obtain 
\begin{theorem}
\label{thm:a=1,c=4-(2)}
{\it
If $f^{a,b}_4$ represents $n=1,2,3,$ 
it can represent all the nonnegative integers. 
}
\end{theorem}

\section{The case where $c\ge 5$}

\subsection{The case where $c=5$}

\begin{theorem}
\label{thm:c=5}
{\it
There exist no positive integers $a,b$ such that 
$1\le a\le b$ and 
$f^{a,b}_5$ can represent any nonnegative integers $n\in\mathbb{N}_0,$ 
where 
$$
f^{a,b}_5(x,y,z,w)=ax^2+by^2+5(z^2+zw+w^2). 
$$
}
\end{theorem}

\begin{proof}
Suppose that there exist such positive integers $a$ and $b.$ 
Considering $n=1,$ we have $a=1.$ 
Taking $n=2,$ we obtain $b=1,2.$ 
Considering $n=3,$ we have $(a,b)=(1,2).$ 
\par
Take $n=10.$ Then, 
$n$ cannot be written as $x^2+2y^2$ with $x,y\in\mathbb{Z},$ 
which implies that $z^2+zw+w^2=1,$ 
because $2$ cannot be represented as $z^2+zw+w^2$ with $z,w\in\mathbb{Z}.$ 
Therefore, 
it follows that $5=10-5\cdot 1$ can be written as $x^2+2y^2$ with $x,y\in\mathbb{Z},$ 
which is impossible.  
\end{proof}

\subsection{The case where $c\ge 6$}

\begin{theorem}
\label{thm:cge6}
{\it
There exist no positive integers $a,b,c$ such that 
$1\le a\le b,$ $c\ge 6$ and 
$f^{a,b}_c$ can represent any nonnegative integers $n\in\mathbb{N}_0$. 
}
\end{theorem}

\begin{proof}
Suppose that there exist such positive integers $a,b$ and $c.$ 
Considering $n=1,$ we have $a=1.$ 
Taking $n=2,$ we obtain $b=1,2.$ 
Considering $n=3,$ we have $(a,b)=(1,2).$ 
On the other hand, 
$n=5$ cannot be written as $x^2+2y^2$ with $x,y\in\mathbb{Z}.$  
\end{proof}

\section{Proof of Theorem \ref{thm:main} }

\begin{proof}
Theorem \ref{thm:main} (1) follows from 
Theorems \ref{thm:c=1-summary}, \ref{thm:a=1,c=2}, \ref{thm:a=1,c=3}, 
\ref{thm:a=1,c=4}, \ref{thm:c=5} and \ref{thm:cge6}. 
Theorem \ref{thm:main} (2) follows from Theorems \ref{thm:c=1-summary}, 
\ref{thm:a=1,c=2-(2)}, \ref{thm:a=1,c=3-(2)}, \ref{thm:a=1,c=4-(2)}, \ref{thm:c=5} and \ref{thm:cge6}. 
\end{proof}


\begin{thebibliography}{10}

\bibitem{Bhargava-Hanke}
M. Bhargava and J. Hanke, 
{\it Universal Quadratic Forms and the 290-Theorem,}
(preprint)

\bibitem{Berndt}
B. C. Berndt, 
{\it Ramanujan's notebooks. Part III.} 
Springer-Verlag, New York, (1991). 

\bibitem{Dickson-1}
L. E. Dickson, 
{\it Integers represented by positive ternary quadratic forms,} 
Bull. Amer. Math. Soc. {\bf 33} 63-70 (1927). 

\bibitem{Dickson-2}
L. E. Dickson, 
{\it Modern Elementary Theory of Numbers,} 
University of Chicago Press, Chicago, (1939).




\bibitem{Hirschhorn-Garvan-Borwein}
M. Hirschhorn, F. Garvan, and J. Borwein, 
{\it
Cubic analogues of the Jacobian theta function $\theta(z,q),$
} 
Canad. J. Math. {\bf 45} 673-694 (1993). 









\bibitem{Jacobi}
C. G. J. Jacobi, 
{\it 
Fundamenta Nova Theoriae Functionum Ellipticarum,
} 
Borntr\"ager, 
Regiomonti, (1829).



\bibitem{Ramanujan}
S. Ramanujan, 
{\it
On the expression of a number in the form $ax^2+by^2+cz^2+du^2,$
} 
Proc. Cambridge Philos. Soc. {\bf 19} 11-21 (1917).



\end{thebibliography}
\end{document}